\documentclass[11pt]{amsart}
\usepackage{amsthm}
\usepackage{amsmath}
\usepackage[margin=1in]{geometry}
\usepackage{amssymb}
\usepackage{verbatim}

\newtheorem{theorem}{Theorem}

\newtheorem{lemma}[theorem]{Lemma}
\newtheorem{corollary}[theorem]{Corollary}
\newtheorem*{conjecture}{Conjecture}
\newtheorem{definition}[theorem]{Definition}

\newtheorem{proposition}[theorem]{Proposition}
\newcommand{\THM}[1]{Theorem \ref{#1}}\newcommand{\THMS}[2]{Theorems \ref{#1} and \ref{#2}}
\newcommand{\SEC}[1]{Section \ref{#1}}\newcommand{\SECS}[2]{Sections \ref{#1} and \ref{#2}}
\newcommand{\LEM}[1]{Lemma \ref{#1}}
\newcommand{\LEMS}[2]{Lemmas \ref{#1} and \ref{#2}}
\newcommand{\COR}[1]{Corollary \ref{#1}}
\newcommand{\PROP}[1]{Proposition \ref{#1}} \newcommand{\PROPS}[2]{Propositions \ref{#1} and \ref{#2}}
\newcommand{\DEF}[1]{Definition \ref{#1}}

\newcommand{\la}{\lambda}

\newcommand{\Z}{\mathbb{Z}}
\newcommand{\Q}{\mathbb{Q}}

\newcommand{\C}{\mathbb{C}}
\newcommand{\HH}{\mathbb{H}}
\newcommand{\embedded}{\hookrightarrow}
\DeclareMathOperator{\cod}{cod}

\newcommand{\floor}[1]{\left\lfloor #1 \right\rfloor}

\newcommand{\of}[1]{\left(#1\right)}

\newcommand{\dk}[1]{\dim\ker\of{#1}}
\newcommand{\lra}{\longrightarrow}

\DeclareMathOperator{\Isom}{Isom}

\DeclareMathOperator{\rank}{rank}

\DeclareMathOperator{\bodd}{b_{\mathrm{odd}}}

\author{Lee Kennard}
\title{On the Hopf conjecture with symmetry}

\begin{document}

\begin{abstract}
The Hopf conjecture states that an even-dimensional, positively curved Riemannian manifold has positive Euler characteristic. We prove this conjecture under the additional assumption that a torus acts by isometries and has dimension bounded from below by a logarithmic function of the manifold dimension. The main new tool is the action of the Steenrod algebra on cohomology.
\end{abstract}
\maketitle

Positively curved spaces have been of interest since the beginning of global Riemannian geometry. Unfortunately, there are few known examples (see \cite{Ziller07} for a survey and \cite{Petersen-Wilhelm09pre, Dearricott11, Grove-Verdiani-Ziller11} for recent examples) and few topological obstructions to any given manifold admitting a positively curved metric. In fact, all known simply connected examples in dimensions larger than 24 are spheres and projective spaces, and all known obstructions to positive curvature for simply connected manifolds are already obstructions to nonnegative curvature.

One famous conjectured obstruction to positive curvature was made by H. Hopf in the 1930s. It states that even-dimensional manifolds admitting positive sectional curvature have positive Euler characteristic. This conjecture holds in dimensions two and four by the theorems of Gauss-Bonnet or Bonnet-Myers (see \cite{Bishop-Goldberg64, Chern56}), but it remains open in higher dimensions.

In the 1990s, Karsten Grove proposed a research program to address our lack of knowledge in this subject. The idea is to study positively curved metrics with large isometry groups. This approach has proven to be quite fruitful (see \cite{Wilking07, Grove09} for surveys). Our main result falls into this category:

\begin{theorem}\label{chi>0}
Let $M^n$ be a closed Riemannian manifold with positive sectional curvature and $n\equiv 0\bmod{4}$. If $M$ admits an effective, isometric $T^r$-action with $r \geq 2\log_2 n$, then $\chi(M)>0$.
\end{theorem}

Previous results showed that $\chi(M^n)>0$ under the assumption of a linear bound on $r$. For example, a positively curved $n$-manifold with an isometric $T^r$-action has positive Euler characterisic if $n$ is even and $r\geq n/8$ or if $n\equiv 0\bmod{4}$ and $r\geq n/10$ (see \cite{Rong-Su05, Su-Wang08}).

Theorem \ref{chi>0} easily implies similar results where the assumption on the symmetry rank, i.e. $\rank(\Isom(M))$, is replaced by one where the symmetry degree, i.e. $\dim(\Isom(M))$, is large or the cohomogeneity, i.e. $\dim(M/\Isom(M))$, is small (see Section \ref{Observations}).

A key tool is Wilking's connectedness theorem (see \cite{Wilking03}), which has proven to be fundamental in the study of positively curved manifolds with symmetry. The theorem relates the topology of a closed, positively curved manifold with that of its totally geodesic submanifolds of small codimension. Since fixed-point sets of isometries are totally geodesic, this becomes a powerful tool in the presence of symmetry.

Part of the utility of the connectedness theorem is to allow proofs by induction over the dimension of the manifold. Another important implication is a certain periodicity in cohomology. By using the action of the Steenrod algebra on cohomology, we refine this periodicity in some cases. For example, we prove:

\begin{theorem}[Periodicity Theorem]\label{periodicityTHM}
Let $N^n$ be a closed, simply connected, positively curved manifold which contains a pair of totally geodesic, transversely intersecting submanifolds of codimensions $k_1\leq k_2$. If $k_1 + 3k_2 \leq n$, then $H^*(N;\Q)$ is $\gcd(4,k_1,k_2)$-periodic.
\end{theorem}

It follows from \cite{Wilking03} that, under these assumptions, $H^*(N;\Q)$ is $\gcd(k_1,k_2)$-periodic. For a closed, orientable $n$-manifold $N$ and a coefficient ring $R$, we say that $H^*(N;R)$ is $k$-periodic if there exists $x\in H^k(N;R)$ such that the map $H^i(N;R)\to H^{i+k}(N;R)$ induced by multiplication by $x$ is surjective for $0\leq i < n-k$ and injective for $0<i\leq n-k$.

To illustrate the strength of the conclusion of \THM{periodicityTHM}, we observe the following:
	\begin{itemize}
	\item If $\gcd(4,k_1,k_2) = 1$, then $N$ is a rational homology sphere.
	\item If $\gcd(4,k_1,k_2) = 2$, then $N$ has the rational cohomology of $S^n$ or $\C P^{n/2}$.
	\item If $\gcd(4,k_1,k_2) = 4$ and $n\not\equiv 2\bmod{4}$, then $N$ has the rational cohomology ring of $S^n$, $\C P^{n/2}$, $\HH P^{n/4}$, or $S^3\times \HH P^{(n-3)/4}$.
	\end{itemize}
When $\gcd(4,k_1,k_2)=4$ and $n\equiv 2\bmod{4}$, the rational cohomology rings of $S^n$, $\C P^{n/2}$, $S^2 \times \HH P^{(n-2)/4}$ and
	\[M^6 = (S^2\times S^4) \#(S^3\times S^3)\# \cdots \#(S^3\times S^3)\]
are 4-periodic, but we do not know whether other examples exist in dimensions greater than six. This uncertainty is what prevents us from proving \THM{chi>0} in all even dimensions (see \SEC{Observations}).

The main step in the proof of Theorem \ref{periodicityTHM} is the following topological result:

\begin{theorem}\label{THM4-newer}
If $M^n$ is a closed, simply connected manifold such that $H^*(M;\Z)$ is $k$-periodic with $3k\leq n$, then $H^*(M;\Q)$ is $\gcd(4,k)$-periodic.
\end{theorem}

To prove \THM{THM4-newer}, we note that the assumption implies the same periodicity with coefficients in $\Z_p$. We then use the action of the Steenrod algebra for $p=2$ and $p=3$ to improve these periodicity statements with coefficients in $\Z_p$. When combined, this information implies $\gcd(4,k)$-periodicity with coefficients in $\Q$. See \PROPS{mod2lemma}{modplemma} in \SECS{SECmod2periodicity}{SECmodpperiodicity} for a more general periodicity statement with coefficients in $\Z_p$, which can be viewed as a generalization of Adem's theorem on singly generated cohomology rings (see \cite{Adem52}).
\smallskip

With the periodicity theorem in hand, we briefly explain some of the tools that go into the proof of \THM{chi>0}. The starting point is a theorem of Lefschetz which states that the Euler characteristic satisfies $\chi(M) = \chi(M^T)$, where $M^T$ is the fixed-point set of the torus action. Since $M$ is even-dimensional with positive curvature, $M^T$ is nonempty by a theorem of Berger. Writing $\chi(M^T) = \sum \chi(F)$ where the sum runs over components $F$ of $M^T$, we see that it suffices to show $\chi(F) > 0$ for all $F$. In fact, we prove that $F$ has vanishing odd Betti numbers. An important tool is a theorem of Conner, which states that, if $P$ is a manifold on which $T$ acts, then $\bodd(P^T) \leq \bodd(P)$, where $\bodd$ denotes the sum of the odd Betti numbers. The strategy is to find a submanifold $P$ on which a subtorus $T'\subseteq T$ acts such that $\bodd(P) = 0$ and such that $F$ is a component of $P^{T'}$.

In order to find such a submanifold $P$, we investigate the web of fixed-point sets of $H\subseteq T$, where $H$ ranges over subgroups of involutions. These fixed-point sets are totally geodesic submanifolds on which $T$ acts, so, under the right conditions, we can induct over dimension. In addition, studying fixed-point sets of involutions has the added advantage that we can easily control the intersection data by studying the isotropy representation at a fixed-point of $T$. 

In order to apply the periodicity theorem, we must find a transverse intersection in the web of fixed-point sets of subgroups of involutions. To strip away complication while preserving the required codimension, symmetry, and intersection data, we define an abstract graph $\Gamma$ where the vertices correspond to involutions whose fixed-point sets satisfy certain codimension and symmetry conditions. An edge exists between two involutions if the intersection of the corresponding fixed-point sets is not transverse. We then break up the proof into several parts, corresponding to the structure of this graph.

\smallskip
Here is a short description of the individual sections. In Sections \ref{SECmod2periodicity} and \ref{SECmodpperiodicity}, we prove the mod 2 and mod $p$ generalizations, respectively, of Adem's theorem on singly generated cohomolgy rings. In \SEC{SECmodQperiodicity}, we bring together the results of the previous two sections to prove \THM{THM4-newer}. In \SEC{SECperiodicitytheorem}, we prove \THM{periodicityTHM}, and in \SEC{ProofTHMchi>0}, we prove \THM{chi>0}. Finally, in \SEC{Observations}, we derive a corollary of \THM{chi>0}, state a periodicity conjecture which would generalize Adams's theorem on singly generated cohomology rings, and we explain how a proof of this conjecture would imply that \THM{chi>0} holds in all even dimensions.

\smallskip This work is part of the author's Ph.D. thesis. The author would like to thank his advisor, Wolfgang Ziller, for his encouragement and numerous suggestions along the way. The author would also like to thank Anand Dessai and Jason DeVito for helpful comments.

\bigskip
\section{Periodicity with coefficients in $\Z_2$}\label{SECmod2periodicity}
\bigskip

We recall the basic definition of periodicity for reference:
\begin{definition}
For a topological space $M$, a ring $R$, and an integer $c$, we say that $x\in H^k(M;R)$ induces periodicity in $H^*(M;R)$ up to degree $c$ if the maps $H^i(M;R) \to H^{i+k}(M;R)$ given by multiplication by $x$ are surjective for $0\leq i<c-k$ and injective for $0<i\leq c-k$.

When such an element $x\in H^k(M;R)$ exists, we say that $H^*(M;R)$ is $k$-periodic up to degree $c$. If in addition $M$ is a closed, orientable manifold and $c = \dim(M)$, then we say that $H^*(M;R)$ is $k$-periodic.
\end{definition}

Observe, for example, that the definition implies $\dim_R H^{ik}(M;R) \leq 1$ for $ik < c$ when $R$ is a field and $M$ is connected. As another example, when $M^n$ is a closed, simply manifold with $k$-periodic $\Z_p$-cohomology ring, then $H^{1+ik}(M;\Z_p) = 0$ and $H^{n-1-ik}(M;\Z_p) = 0$ for all $i$.

We remark that $H^*(M;R)$ is trivially $k$-periodic up to degree $c$ when $k\geq c$. By a slight abuse of notation, we also say that $H^*(M;R)$ is $k$-periodic if $2k\leq c$ and $H^i(M;R)=0$ for $0<i<c$. One thinks of 0 as the element inducing periodicity. This convention simplifies the discussion.

We start with a general lemma about periodicity:

\begin{lemma}\label{modplemma1}
Let $R$ be a field. If $x\in H^k(M;R)$ is a nonzero element inducing periodicity up to degree $c$ with $2k\leq c$, and if $x^r = yz$ for some $1\leq r\leq c/k$ with $\deg(y) \not\equiv 0 \bmod{k}$, then $y$ also induces periodicity.

In particular, if $x = yz$ with $0<\deg(y)<k$, then $y$ induces periodicity.
\end{lemma}

The way in which we will use this lemma is to take an element $x$ of minimal degree that induces periodicity, and to conclude that the only factorizations of $x^r$ are those of the form $(ax^s)(bx^{t})$ where $a,b\in R$ are multiplicative inverses and $r = s+t$.
\begin{proof}
Use periodicity to write $y=y'y''$ and $z=z'z''$ where $y''$ and $z''$ are powers of $x$, $0<\deg(y')<l$, and $0<\deg(z')<l$. Since $x$ generates $H^k(M;R)$, it follows that $y'z' = ax$ for some multiple $a\in R$. If $a=0$, then $x^r = (ax)y''z'' = 0$, a contradiction to periodicity and the assumption that $x\neq 0$. Supposing therefore that $a\neq 0$, we may multiply by $a^{-1}$ to assume without loss of generality that $x = y'z'$. Since $y''$ is a multiple of $x$, we note that it suffices to show that $y'$ induces periodicity up to degree $c$. Let $k' = \deg(y')$.

Since multiplying by $x$ is injective from $H^i(M;\Z_2)\to H^{i+k}(M;\Z_2)$, and since this map factors as multiplication by $y'$ followed by multiplication by $z'$, it follows that multiplication by $y'$ from $H^i(M;\Z_2)\to H^{i+k'}(M;\Z_2)$ is injective for $0<i\leq c-k$. In addition, to see that multiplication by $y'$ is injective from $H^i(M;\Z_2)\to H^{i+k'}(M;\Z_2)$ for $c-k < i \leq c-k'$, consider that multiplying by $y'$ and then by $x$ from
	\[H^{i-k}(M;\Z_2) \to H^{i-k+k'}(M;\Z_2) \to H^{i+k'}(M;\Z_2)\]
is the same as multiplying by $x$ and then by $y'$ from
	\[H^{i-k}(M;\Z_2) \to H^{i}(M;\Z_2) \to H^{i+k'}(M;\Z_2).\]
Since the first composition is injective, and since the first map in the second is an isomorphism, we conclude that multiplication by $y'$ is injective from $H^i(M;\Z_2)\to H^{i+k'}(M;\Z_2)$ for all $0<i\leq c-k'$. The proof that multiplication by $y'$ is surjective in all required degrees is similar.
\end{proof}

In \cite{Adem52}, Adem showed that, for a topological space $M$, if $H^*(M;\Z_2)$ is isomorphic to $\Z_2[x]$ or $\Z_2[x]/x^{q+1}$ with $q\geq 2$, then then $k = \deg(x)$ is a power of 2. Observe that such a cohomology ring is $k$-periodic and that $k$ is the minimal period. We now prove the following generalization of Adem's theorem:

\begin{proposition}[$\Z_2$-periodicity theorem]\label{mod2lemma}
Suppose $x\in H^l(M;\Z_2)$ is nonzero and induces periodicity in $H^*(M;\Z_2)$ up to degree $c$ with $2l\leq c$. If $x$ has minimal degree among all such elements, then $l$ is a power of 2.
\end{proposition}

The key tool in the proof is the existence of Steenrod squares, so we review some of their properties now. The Steenrod squares are group homomorphisms
	\[Sq^i : H^*(M;\Z_2) \to H^*(M;\Z_2)\]
which exist for all $i\geq 0$ and satisfy the following properties:
	\begin{enumerate}
	 \item If $x\in H^j(M;\Z_2)$, then $Sq^i(x) \in H^{i+j}(M;\Z_2)$, and
	 	\begin{itemize}
	 	 \item if $i=0$, then $Sq^i(x) = x$,
	 	 \item if $i=j$, then $Sq^i(x) = x^2$, and
	 	 \item if $i>j$, then $Sq^i(x) = 0$.
	 	\end{itemize}
	 \item (Cartan formula) If $x,y\in H^*(M;\Z_2)$, then $Sq^i(xy) = \sum_{0\leq j\leq i} Sq^j(x)Sq^{i-j}(y)$.
	 \item (Adem relations) For $a<2b$, one has the following relation among compositions of Steenrod squares: 
	 	\[Sq^a Sq^b = \sum_{j=0}^{\floor{a/2}} \binom{b-1-j}{a-2j} Sq^{a+b-j} Sq^j.\]
	\end{enumerate}
A consequence of the Adem relations is the following: If $l$ is not a power of two, there exists a relation of the form $Sq^l = \sum_{0<i<l} a_i Sq^i Sq^{l-i}$ for some constants $a_i$. Indeed, if $l = 2^c + d$ for integers $c$ and $d\equiv 0 \bmod{2^{c+1}}$, then we can use the Adem relation with $(a,b) = (2^c, d)$.

The first application of the Steenrod squares in the presence of periodicity is to show the following:
\begin{lemma}\label{mod2lemma2}
Suppose $x\in H^l(M;\Z_2)$ is nonzero and induces periodicity up to degree $c$ with $2l\leq c$. If $x = Sq^i(y)$ for some $i>0$, then $x$ factors as a product of elements of degree less than $l$.
\end{lemma}
Combined with \LEM{modplemma1}, we conclude that if $i>0$ and if $x=Sq^i(y)\in H^l(M;\Z_2)$ is nonzero and induces periodicity up to degree $c$ with $2l\leq c$, then there is another nonzero element $x'$ inducing periodicity up to degree $c$ with $0<\deg(x') < l$.
\begin{proof}
Let $i>0$ be maximal such that $x = Sq^i(y)$ for some cohomology element $y$. Using the Cartan relation, we compute $x^2$ as follows:
	\[x^2 = Sq^i(y)^2 = Sq^{2i}(y^2) - \sum_{j\neq i}Sq^j(y)Sq^{2i-j}(y).\]
Now $Sq^j(y)$ and $Sq^{2i-j}(y)$ commute, so the sum over $j\neq i$ is twice the sum over $j<i$. Hence $x^2 = Sq^{2i}(y^2)$.

Next, $Sq^i(y)=x\neq 0$ implies $i\leq \deg(y)$. Moreover, $i = \deg(y)$ implies that $x$ factors as $y^2$. Suppose then that $i < \deg(y)$. Since $l = i + \deg(y)<\deg(y^2)$, it follows from the surjectivity assumption of periodicity that $y^2 = xy'$ for some $y'$ with $0<\deg(y')<\deg(y)$. Using periodicity again, observe that $Sq^j(x)$ for $0\leq j<l$ can be factored as $xx_j$ for some $x_j \in H^j(M;\Z_2)$. Applying the Cartan formula again, we have
	\[x^2 = Sq^{2i}(xy') = \sum_{j\leq 2i} xx_jSq^{2i-j}(y').\]
The injectivity assumption of periodicity implies we may cancel an $x$ and conclude $x=x_jSq^{2i-j}(y')$ for some $j\leq 2i$. Because $i$ was chosen to be maximal, we must have $j>0$, that is, we must have that $x$ factors as a product of elements of degree less than $l$.
\end{proof}

We proceed to the proof of \PROP{mod2lemma}. Suppose $x\in H^*(M;\Z_2)$ is nonzero, induces periodicity up to degree $c$ with $2l\leq c$, and has minimal degree among all such elements. Assume $l$ is not a power of 2. We will show that $x$ factors nontrivially or that $x=Sq^i(y)$ for some $i>0$, which contradicts \LEMS{modplemma1}{mod2lemma2}.

The first step is to evaluate the Adem relation $Sq^l = \sum_{0<i<l} a_i Sq^i Sq^{l-i}$ on $x$. Using the factorization $Sq^j(x) = xx_j$ as above, together with the Cartan formula, we obtain
 	\[x^2 = Sq^l(x) = \sum_{0<i<l} a_i Sq^i(xx_{l-i}) = \sum_{0<i<l} a_i \sum_{0\leq j\leq i} xx_jSq^{i-j}(x_{l-i}).\]
Using the injectivity assumption of periodicity, we can cancel an $x$ to conclude that
 	\[x = \sum_{0<i<l} \sum_{0\leq j \leq i} a_ix_jSq^{i-j}(x_{l-i}).\]
Now periodicity and our assumption that $x$ is nonzero imply that $H^l(M;\Z_2)\Z_2$ and is generated by $x$. It follows that $x = x_jSq^{i-j}(x_{l-i})$ for some $0<i<l$ and $0\leq j\leq i$. If $j>0$, we have proven a nontrivial factorization of $x$, and if $j=0$, we have proven that $x = Sq^i(x_{l-i})$ for some $i>0$. As explained at the beginning of the proof, this is a contradiction. 

\bigskip
\section{Periodicity with coefficients in $\Z_p$}\label{SECmodpperiodicity}
\bigskip

In this section, we prove the $\Z_p$-analogue of \PROP{mod2lemma}:
 
\begin{proposition}[$\Z_p$-periodicity theorem]\label{modplemma}
Let $p$ be an odd prime. Suppose $x\in H^l(M;\Z_p)$ is nonzero and induces periodicity in $H^*(M;\Z_p)$ up to degree $c$ with $pl\leq c$. If $x$ has minimal degree among all such elements, then $l=2\lambda p^r$ for some $r\geq 0$ and $\lambda\mid p-1$.
\end{proposition}

The proof uses Steenrod powers. These are group homomorphisms
	\[P^i : H^*(M;\Z_p) \to H^*(M;\Z_p)\]
for $i\geq 0$ that satisfy the following properties:
	\begin{enumerate}
	 \item If $x\in H^j(M;\Z_p)$, then $P^i(x) \in H^{j+2i(p-1)}(M;\Z_p)$, and 
	 	\begin{itemize}
	 	 \item if $i=0$, then $P^i(x) = x$,
	 	 \item if $2i=j$, then $P^i(x) = x^p$, and
	 	 \item if $2i>j$, then $P^i(x) = 0$.
	 	\end{itemize}
	 \item (Cartan formula) For $x,y\in H^*(M;\Z_p)$, $P^i(xy) = \sum_{0\leq j\leq i} P^j(x)P^{i-j}(y)$.
	 \item (Adem relations) For $a<pb$,
	 	\begin{equation}\label{SteenrodPowersAxiom3}
	 	 P^a P^b = \sum_{j\leq a/p} (-1)^{a+j}\binom{(p-1)(b-j)-1}{a-pj}P^{a+b-j}P^j.
	 	\end{equation}
	\end{enumerate}

Despite the similarity of the statements of \PROPS{mod2lemma}{modplemma}, the proof in the odd prime case is more involved. We proceed with a sequence of steps. 

We first study the structure of the Adem relations to obtain a specific relation in the $\Z_p$-algebra $\mathcal{A}$ generated by $\{P^i\}_{i\geq 0}$ modulo the Adem relations. This lemma does not use periodicity.

\begin{lemma}\label{modplemma4}
Let $k = \lambda p^a + \mu$ where $0<\lambda < p$ and $\mu\equiv 0\bmod{p^{a+1}}$. For all $1\leq m\leq \lambda$, there exist $Q_i\in\mathcal{A}$ such that $P^k = P^{mp^a}\circ Q_a + \sum_{i < a} P^{p^i}\circ Q_i.$
\end{lemma}

\begin{proof}
We induct over $k$. For $k=1$, the result is trivial. Suppose the result holds for all $k'<k$. Write $k = \lambda p^a + \mu$ where $0<\lambda < p$ and $\mu\equiv 0\bmod{p^{a+1}}$, and let $1\leq m\leq \lambda$. If $\mu=0$ and $m=\lambda$, then $P^k = P^{mp^a}$ is already of the desired form. If not, then $mp^a < p(k - mp^a)$ and we have the Adem relation (see Equation \ref{SteenrodPowersAxiom3})
  \[c_0 P^k = P^{mp^a} P^{k-mp^a} - \sum_{0<j\leq mp^{a-1}} c_j P^{k-j} P^j.\]
For $0<j\leq mp^{a-1}$, $k-j$ is less than $k$ and not congruent to 0 modulo $p^a$. Hence the induction hypothesis implies that each $P^{k-j}$ term is of the form $\sum_{i<a} P^{p^i}Q_i$. It therefore suffices to prove that $c_0 \not\equiv 0\bmod{p}$. 

For this, we use the following elementary fact: If $x = \sum_{i\geq 0} x_i p^i$ and $y=\sum_{i\geq 0} y_i p^i$ are base $p$ expansions (where $p$ is prime), then the modulo $p$ binomial coefficients satisfy $\binom{x}{y} \equiv \prod \binom{x_i}{y_i} \bmod{p}$. Hence we have
  \[(-1)^{m}c_0 = 	\binom{(p-1)(k-mp^a)-1}{mp^a}
		  \equiv\binom{(p-1)(\lambda-m)p^a-1}{mp^a}
		  \equiv\binom{p-(\lambda-m)-1}{m},
  \]
which is not congruent to 0 modulo $p$ since $0\leq m \leq p-(\lambda-m)-1 < p$. This completes the proof.
\end{proof}

To simplify the remainder of the proof, we assume throughout the rest of the section that $x\in H^l(M;\Z_p)$ is nonzero, induces periodicity up to degree $c$ with $pl \leq c$, and has minimal degree among all such elements. In particular, \LEM{modplemma1} implies that the only factorizations of $ax^r$ with $a\neq 0$ and $r\leq p$ are of the form $(a'x^{r'})(a'' x^{r''})$ with $a',a''\in \Z_p$ and $r = r' + r''$. The next step is to prove an analogue of \LEM{mod2lemma2}:

\begin{lemma}\label{modplemma2}
No nontrivial multiple of $x$ is of the form $P^i(y)$ with $i>0$.
\end{lemma}

\begin{proof}
 Without loss of generality, we may assume $x$ itself is equal to $P^i(y)$ for some $i>0$. Set $d=\deg(y)$. Our first task is to write some power of $x$ as $P^{i_1}(y_1)$ with $0<\deg(y_1)<d$. Because $x\neq 0$, we have $2i\leq d$, which implies $\deg(y^p) \geq l$. Let $r$ be the integer such that $l+d>\deg(y^r)\geq l$. Lemma \ref{modplemma1} implies a strict inequality here. Using periodicity, write $y^r = xy_1$ with $0<d_1<d$.

 We next calculate the $r$-th power of both sides of the equation $x=P^i(y)$ using the Cartan relation:
  \[x^r = P^{ri}(y^r) - \sum c_{j_1,\ldots,j_r} P^{j_1}(y)\cdots P^{j_r}(y)\]
 where the $c_{j_1,\ldots,j_r}$ are constants and the sum runs over $j_1\geq\cdots\geq j_r$ with $j_1+\ldots+j_r=ri$ and $(j_1,\ldots,j_r)\neq(i,\ldots,i)$. Observe that $j_1>i$, so $P^{j_1}(y) = xz_{j_1}$ for some $z_{j_1}\in H^{2(j_1-i)(p-1)}(M;\Z_p)$. Using $y^r=xy_1$, the first term on the right-hand side becomes
  \[P^{ri}(y^r) = P^{ri}(xy_1) = \sum_{k'+k = ri} P^{k'}(x)P^{k}(y_1) = x\sum_{k'+k=ri}x_{k'}P^{k}(y_1)\]
 for some $x_{k'} \in H^{2k'(p-1)}(M;\Z_p)$. Combining these calculations, and using periodicity to cancel the $x$, we obtain
  \[x^{r-1} = \sum_{k'+k = ri} x_{k'} P^{k}(y_1) + \sum c_{j_1,\ldots,j_r} z_{j_1}P^{j_2}(y)\cdots P^{j_r}(y).\]
 Now $r-1 < p$, so periodicity implies that $x^{r-1}$ generates $H^{(r-1)l}(M;\Z_p)$. By Lemma \ref{modplemma1} therefore, every term of the form $z_{j_1}P^{j_2}(y)\cdots P^{j_r}(y)$ vanishes since $0<\deg(P^{j_r}(y)) < \deg(P^i(y)) = l$. Similarly, all terms of the form $x_{k'} P^{k}(y_1)$ vanish unless $P^{k}(y_1)$ is a power of $x$. Hence some power of $x$ is of the form $P^{i_1}(y_1)$, as claimed.

We now show that, given an expression $x^{r_j} = P^{i_j}(y_j)$ for some $j\geq 1$ with $0<\deg(y_j)< d$, there exists another expression $x^{r_{j+1}} = P^{i_{j+1}}(y_{j+1})$ with $0<\deg(y_{j+1}) < d$. Moreover, it will be apparent that $l + \deg(y_{j+1}) = \deg(y_j) + m_j d$ for some integer $m_j$. First, among all such expressions $x^{r_j} = P^{i_j}(y_j)$, fix $y_j$ and take $r_j$ (or, equivalently, $i_j$) to be minimal. Next, note that $P^{i_j}(y_j)=x^{r_j}\neq 0$ implies $p\deg(y_j) \geq r_jl$, which together with $pd \geq l$ implies
    \[p\deg(y_j y^{p-r_j}) = p\deg(y_j) + (p-r_j)pd \geq pl.\]
 Hence we can choose an integer $m_j\leq p-r_j$ satisfying $l \leq \deg(y_j) + m_jd < l + d$. Once again, Lemma \ref{modplemma1} implies both inequalities are strict. Using periodicity, we can write $y_j y^{m_j} = xy_{j+1}$ with $0<\deg(y_{j+1})<d$ and $l + \deg(y_{j+1}) = \deg(y_j) + m_jd$. We now calculate
  \[x^{r_j + m_j} = P^{i_j}(y_j) P^i(y)^{m_j} = P^{i_j+m_ji}(y_jy^{m_j}) - \sum P^{k_0}(y_j)P^{k_1}(y)\cdots P^{k_{m_j}}(y)\]
 where the sum runs over $(k_0,\ldots,k_{m_j})\neq (i_j,i,\ldots,i)$ with $k_0+\ldots+k_{m_j}=i_j + m_j i$. As when we calculated $x^r$ above, we are able to factor an $x$ from each term on the right-hand side and use periodicity to cancel it. Using that $x^{r_j+m_j - 1}$ is a generator and Lemma \ref{modplemma1}, together with the assumption that $r_j$ is minimal, we conclude that $x$ is a nonzero multiple of $P^{i_{j+1}}(y_{j+1})$, as claimed.

 We therefore have a sequence of cohomology elements $y_1,y_2,\ldots$ with $0<\deg(y_j)<d$ and $l + \deg(y_{j+1}) = \deg(y_j) + m_jd$ for some integer $m_j$ for all $j\geq 1$. This cannot be. Indeed, adding the equations $l+\deg(y_1) = rd$ and $l+\deg(y_{j+1}) = \deg(y_j) + m_jd$ for $1\leq j \leq d - 1$ yields
  \[l d + \deg(y_d) = (r+m_1+\ldots+m_{d-1})d,\]
 which implies that $\deg(y_d)$ is divisible by $d$. But $0<\deg(y_d)<d$, so this is a contradiction.
\end{proof}

\LEM{modplemma2} easily implies the following:

\begin{lemma}\label{modplemma3}
No nontrivial multiple of $x^r$ with $1\leq r\leq p$ is of the form $P^i(y)$ with $0<i<\frac{l}{2(p-1)}$.
\end{lemma}

\begin{proof}
Indeed, the bound on $i$ implies $\deg(y) = ml - 2i(p-1) > (m-1)l$, so periodicity implies $y = x^{m-1}z$ for some $z$ with $0<\deg(z)<l$. Applying the Cartan formula and periodicity, we obtain $x^m = P^j(x^{m-1})P^{j'}(z)$ for some $j+j'=i$. By Lemma \ref{modplemma1}, $P^{j'}(z)$ is a power of $x$. But since $\deg(P^j(x^{m-1}))\geq (m-1)l$, we must have $x = P^{j'}(z)$. Since $\deg(z) < l$, we have a contradiction to Lemma \ref{modplemma2}.
\end{proof}

At this point, we combine what we have established so far. Recall that we are assuming $x\in H^l(M;\Z_p)$ is nonzero and induces periodicity up to degree $2l\leq c$ and that $x$ has minimal degree among all such elements. Observe that $p>2$ implies $x^3\neq 0$. Hence $l=2k$ for some $k$.

\begin{lemma}\label{modplemma5}
Suppose $l=2k$ and $k=\lambda p^a + \mu$ for some $0<\lambda<p$ and $\mu\equiv 0\bmod{p^{a+1}}$. For all $1\leq m \leq \la$, there exists $r<p$, $0<j\leq mp^a$, and $z\in H^{2p^a(r\la - m(p-1))}(M;\Z_p)$ such that $x^r = P^j(z)$.

Moreover, $j\equiv 0 \bmod{p^a}$ and $0\leq\deg(z)<l$ with $\deg(z)=0$ only if $rk = (p-1)mp^a$.
\end{lemma}

\begin{proof}
Let $l$, $k$, and $m$ be as in the assumption. Evaluating the expression in \LEM{modplemma4} on $x$ yields
  \[x^p = P^k(x) = P^{mp^a}(Q_a(x)) + \sum_{i < a} P^{p^i}(Q_i(x)).\]
Using periodicity, we can write $Q_i(x) = x z_i$ for $i<a$, $Q_a(x) = x^{p-r}z$ for some $1\leq r<p$ such that $0\leq \deg(z) < l$, $P^j(x)=xy_i$ for all $j$, and $P^{mp^a-j}(x^{p-r})=x^{p-r} w_j$ for all $j$.

Using this notation and the Cartan formula, we have
  \[x^p = x^{p-r} \sum_{j\leq mp^a} w_j P^j(z) + x \sum_{i<a} \sum_{j\leq i} y_j P^{p^i-j}(z_i).\]
Using periodicity again, we obtain
  \[x^{p-1} = x^{p-r-1}\sum_{j\leq mp^a} w_j P^j(z) + \sum_{i<a}\sum_{j\leq i} y_j P^{p^i-j}(z_i).\]
Periodicity implies $x^{p-1}$ is a (nonzero) generator of $H^{(p-1)l}(M;\Z_p)$, hence a nontrivial multiple of $x^{p-1}$ is $x^{p-r-1} w_j P^j(z)$ for some $j\leq mp^a$ or $y_j P^{p^i-j}(z_i)$ for some $j\leq i<a$. In the second case, we have a contradiction to \LEM{modplemma1} or \LEM{modplemma3} since $p^i-j \leq p^{a-1} < l/2(p-1)$. Similarly, we have a contradiction to \LEM{modplemma1} in the first case unless $w_j$ is a power of $x$. Moreover, $\deg(w_j) = 2(p-1)(mp^a-j)$ implies $j=ip^a$ for some $0\leq i \leq m$.

Using periodicity to cancel powers of $x$, we conclude that $x^r = P^j(z)$ for some $r$, $j$, and $z$ as in the conclusion of the lemma.
\end{proof}

Using this result, we are in a position to prove \PROP{modplemma}, that is, we are ready to show that $l = 2\lambda p^r$ for some $r\geq 0$ and some $\lambda|p-1$: 

\begin{proof}[Proof of \PROP{modplemma}]
Suppose $x\in H^l(M;\Z_p)$ is nonzero, induces periodicity up to degree $c$ with $pl\leq c$, and has minimal degree among all such elements.

Since $p>2$, we have $l=2k$ for some $k$. Write $k=\lambda p^a + \mu$ where $0<\lambda < p$ and $\mu \equiv 0 \bmod{p^{a+1}}$. Let $g = \gcd(\lambda, p-1)$. Our task is to show that $\mu = 0$ and that $g=\la$. We prove this by contradiction using three cases.

Suppose first that $\mu>0$. Take $m=1$ in \LEM{modplemma5}. Then we have $x^r = P^{p^a}(z)$ with $\deg(z)>0$ since 
	\[rk \geq k \geq \mu \geq p^{a+1} > (p-1)mp^a.\]
Because
  \[p^a = (p^a + p^{a+1})/(p+1) < l/2(p-1),\]
we have a contradiction to \LEM{modplemma3}.

Second, suppose that $\mu=0$ and $1=g<\la$. Choose $1\leq m < \la$ such that $m(p-1) \equiv -1 \mod{\la}$. \LEM{modplemma5} implies the existence of $r<p$, $0<j\leq mp^a$ with $j\equiv 0\bmod{p^a}$, and $z\in H^{2p^a(r\la - m(p-1))}(M;\Z_p)$ with $0\leq \deg(z)<l$ such that $x^r=P^j(z)$. Our choice of $m$ and the conditions on $\deg(z)$ imply $\deg(z) = 2p^a$. In addition, $x^r\neq 0$ implies $j\leq p^a$, so the conditions on $j$ imply $j=p^a$. Putting these facts together implies $x^r = P^j(z)= z^p$. Because $0<\deg(z)<l$, this contradicts \LEM{modplemma1}.

Finally, suppose that $\mu = 0$ and $1 < g < \la$. Taking $m=1$ yields $x^r=P^{p^a}(z)$ with $\deg(z)>0$ since $g<\lambda$ implies
	\[(p-1)mp^a = (p-1)p^a \neq r\lambda p^a =rk.\]
Raising both sides to the $(\lambda/g)$-th power, we obtain
  \[x^{r\lambda/g} = P^{\lambda p^a/g}(z^{\lambda/g}) - \sum P^{i_1}(z)\cdots P^{i_{\lambda/g}}(z)\]
where the sum runs over $i_1 + \ldots + i_{\lambda/g} = \lambda p^a/g$ with $(i_1,\ldots,i_{\lambda/g})\neq(p^a, \ldots, p^a)$. Observe that $\deg(z^{\lambda/g})$ is a multiple of $l=\deg(x)$ while $0<\deg(z)<l$, so Lemma \ref{modplemma1} implies $z^{\lambda/g}=0$. Now $g\geq 2$ and $rl-2(p-1)p^a = \deg(z) < l$ implies $r\lambda/g < p$, so $x^{r\lambda/g}$ is nonzero and generates $H^{lr\lambda/g}(M;\Z_p)$. This implies $x^{r\lambda/g}$ is a nontrivial multiple of $P^{i_1}(z)\cdots P^{i_{\lambda/g}}(z)$ for some $(i_1,\ldots,i_{\lambda/g})$. Using Lemma \ref{modplemma1} again, we conclude that each $P^{i_j}(z)$ is a power of $x$. But the degrees of $x$ and $z$ implies that this is only the case if $i_j\geq p^a$ for all $j$. Since there is no such term in the sum, we obtain a contradiction.
\end{proof}

\bigskip
\section{Proof of \THM{THM4-newer}}\label{SECmodQperiodicity}
\bigskip

In this section, we use \PROPS{mod2lemma}{modplemma} to prove \THM{THM4-newer} in the introduction. We are given a closed, simply connected manifold $M^n$ and an element $x\in H^k(M;\Z)$ inducing periodicity with $3k\leq n$. Note that if $x$ is a torsion element, then $M$ is a rational homology sphere. Since $H^*(M;\Q)$ is then trivially $\gcd(4,k)$-periodic, we may assume $x$ is not a torsion element. By periodicity, $H^k(M;\Z)\cong \Z$ and is generated by $x$.

Consider the nonzero image $x_2\in H^k(M;\Z_2)$ of $x$ under the reduction homomorphism $H^k(M;\Z)\to H^k(M;\Z_2)$. It follows by the Bockstein sequence
  \[\cdots \stackrel{\cdot 2}{\lra} H^i(M;\Z) \stackrel{\rho}{\lra} H^i(M;\Z_2) \lra H^{i+1}(M;\Z) \lra\cdots\]
and the five lemma that $x_2$ induces periodicity in $H^*(M;\Z_2)$. Denote by $y\in H^l(M;\Z_2)$ the element of minimal degree which induces periodicity. \PROP{mod2lemma} says $l$ is a power of 2.

We claim that $l$ divides $k$. If not, there exists a cohomology element $y'$ with $0<\deg(y')<l$ such that $y^m = x_2 y'$ for some integer $m$. By \LEM{modplemma1}, it follows that $y'$ also induces periodicity, a contradiction to the minimality of $l$. We have then that $l\mid k$. Moreover, periodicity implies $y^{k/l} = x_2$.

Next we show that $y$ comes from an integral element $\tilde{y}\in H^l(M;\Z)$ such that the map $H^i(M;\Z)\to H^{i+l}(M;\Z)$ induced by multiplication by $\tilde{y}$ has finite kernel for all $0<i<n$. Let $\rho:H^l(M;\Z)\to H^l(M;\Z_2)$ be the map induced by reduction modulo 2. Consider first that (via multiplication by $y$) $0=H^1(M;\Z_2) \cong H^{1+l}(M;\Z_2)$, which implies $b_{l+1}(M)\leq b_{l+1}(M;\Z_2)=0$. Next consider the following portion
 	\[H^l(M;\Z) \to H^l(M;\Z_2) \to H^{l+1}(M;\Z) \to H^{l+1}(M;\Z) \to H^{l+1}(M;\Z_2)\]
of the Bockstein sequence. We see that that $H^{l+1}(M;\Z)\to H^{l+1}(M;\Z)$ is a surjection and hence an isomorphism since $H^{l+1}(M;\Z)$ is finite. Using exactness again, we conclude $\rho:H^l(M;\Z)\to H^l(M;\Z_2)$ is surjective, so that we can choose some $\tilde{y}\in H^l(M;\Z)$ with $\rho(\tilde{y})=y$. Now $H^k(M;\Z)$ is generated by $x$, so $\tilde{y}^{k/l} = mx$ for some $m\in\Z$. Applying $\rho$ to both sides yields
 	\[mx_2 = \rho(\tilde{y}^{k/l}) = y^{k/l} =x_2\neq 0,\]
hence $m\neq 0$. This proves that multiplication by $\tilde{y}$ has finite kernel.

Moving to rational coefficients, we conclude that $\bar{y}\in H^l(M;\Q)$, the image of $\tilde{y}$ under the coefficient map $H^l(M;\Z)\to H^l(M;\Q)$, induces periodicity in $H^*(M;\Q)$.

A completely analogous argument using \PROP{modplemma} with $p=3$ shows that $H^*(M;\Q)$ is $m$-periodic with $m = 4\cdot 3^s$. Taking $m$ to be minimal, it follows again that $m\mid k$. At this point it is clear that the Betti numbers of $M$ are $\gcd(k,l,m)$-periodic and hence $\gcd(4,k)$-periodic.

To conclude that $H^*(M;\Q)$ is $\gcd(4,k)$-periodic, consider the set $D$ of all positive integers $d$ such that $H^*(M;\Q)$ has an element in degree $d$ which induces periodicity. Clearly $k,l,m\in D$. We claim that $d_1,d_2\in D$ with $d_1>d_2$ implies $d_1 - d_2\in D$. Indeed suppose $z_1\in H^{d_1}(M;\Q)$ and $z_2\in H^{d_2}(M;\Q)$ induce periodicity in $H^*(M;\Q)$. Since $z_2$ induces periodicity, there exists $z_3\in H^{d_1-d_2}(M;\Q)$ such that $z_1 = z_2z_3$. Since $z_1$ induces periodicity, \LEM{modplemma1} implies that $z_3$ does as well. Since the difference of any two elements in $D$ lies in $D$, it follows that $\gcd(k,l,m)$, and hence $\gcd(4,k)$, also lies in $D$.

\bigskip
\section{Proof of \THM{periodicityTHM}}\label{SECperiodicitytheorem}
\bigskip

The starting point of the proof is the following theorem of Wilking:
\begin{theorem}[Connectedness Theorem, \cite{Wilking03}]\label{THMconnectednesstheorem}
Suppose $M^n$ is a closed Riemannian manifold with positive sectional curvature.
 \begin{enumerate}
  \item If $N^{n-k}$ is connected and totally geodesic in $M$, then $N\embedded M$ is $(n- 2k + 1)$-connected. 
  \item If $N_1^{n-k_1}$ and $N_2^{n-k_2}$ are totally geodesic with $k_1\leq k_2$, then $N_1\cap N_2\embedded N_2$ is $(n - k_1 - k_2)$-connected.
 \end{enumerate}
\end{theorem}

Recall an inclusion $N\embedded M$ is called $h$-connected if $\pi_i(M,N)=0$ for all $i\leq h$. It follows from the relative Hurewicz theorem that the induced map $H_i(N;\Z) \to H_i(M;\Z)$ is an isomorphism for $i<h$ and a surjection for $i=h$. The following is a topological consequence of highly connected inclusions of closed, orientable manifolds (see \cite{Wilking03}):

\begin{theorem}
Let $M^n$ and $N^{n-k}$ be closed, orientable manifolds. If $N\embedded M$ is $(n-k-l)$-connected with $n-k-2l>0$, then there exists $e\in H^k(M;\Z)$ such that the maps $H^i(M;\Z)\to H^{i+k}(M;\Z)$ given by $x\mapsto ex$ are surjective for $l\leq i<n-k-l$ and injective for $l<i\leq n-k-l$.
\end{theorem}

Combining these results with \THM{THM4-newer}, we will prove in this section the following slightly stronger version of Theorem \ref{periodicityTHM}.

\begin{theorem}\label{THMstrongperiodicity}
Let $N^n$ be a closed, simply connected Riemannian manifold with positive sectional curvature. Let $N_1^{n-k_1}$ and $N_2^{n-k_2}$ be totally geodesic, transversely intersecting submanifolds with $k_1\leq k_2$.
	\begin{enumerate}
	\item If $k_1 + 3k_2 \leq n$, then the rational cohomology rings of ${N}$, ${N_1}$, ${N_2}$, and ${N_1\cap N_2}$ are $\gcd(4,k_1,k_2)$-periodic.
	\item If $2k_1 + 2k_2 \leq n$, then the rational cohomology rings of ${N}$, ${N_1}$, ${N_2}$, and ${N_1\cap N_2}$ are $\gcd(4,k_1)$-periodic.
	\item If $3k_1 + k_2 \leq n$ and if $N_2$ is simply connected, then the rational cohomology rings of ${N_2}$ and ${N_1\cap N_2}$ are $\gcd(4,k_1)$-periodic.
	\end{enumerate}
\end{theorem}

We make two remarks. First, all three codimension assumptions imply that $N_1$ is simply connected and that $N_1\cap N_2$ is simply connected if $N_2$ is. This follows by the connectedness theorem since the bounds on the codimensions imply that the inclusions $N_1\embedded N$ and $N_1\cap N_2 \embedded N_2$ induce isomorphisms of fundamental groups. Similarly the first two assumptions imply that $N_2$ is simply connected, but the third condition does not. 

Second, in the proof of \THM{chi>0}, we will only use the following consequence of \THM{THMstrongperiodicity}:

\begin{corollary}\label{CORtoPerThm}
Let $N^n$ be a closed, positively curved manifold with $n\equiv 0\bmod{4}$. Let $N_1^{n-k_1}$ and $N_2^{n-k_2}$ be totally geodesic, transversely intersecting submanifolds with $2k_1 + 2k_2 \leq n$. Then $\bodd(N) = \sum b_{1+2i}(N) = 0$.
\end{corollary}
\begin{proof}[Proof of Corollary \ref{CORtoPerThm}]
Let $\pi: \tilde{N} \to N$ denote the universal Riemannian covering. The submanifolds $\pi^{-1}(N_i)\subseteq N$ are transversely intersecting, totally geodesic, $(n-k_i)$-dimensional submanifolds of the closed, simply connected, positively curved manifold $\tilde{N}$. Since $2k_1 + 2k_2 \leq n$, \THM{THMstrongperiodicity} implies $H^*(\tilde{N};\Q)$ is 4-periodic.

Observe that 4-periodicity and Poincar\'e duality imply $\bodd(\tilde{N})=0$ since $\pi_1(\tilde{N})=0$ and $n\equiv 0 \bmod{4}$. Recall now that the transfer theorem implies that $H^*(N;\Q)$ is isomorphic to $H^*(\tilde{N};\Q)^{\pi_1(N)}$, the subring of invariant elements under the action of $\pi_1(N)$ on $H^*(\tilde{N}; \Q)$. Since $\bodd(\tilde{N}) = 0$, it follows that $\bodd(N) = 0$.
\end{proof}

We proceed now to the proof of \THM{THMstrongperiodicity}. Recall that we have a closed, simply connected Riemannian manifold $N^n$ with positive sectional curvature. We also have totally geodesic, transversely intersecting submanifolds $N_1^{n-k_1}$ and $N_2^{n-k_2}$ with $k_1\leq k_2$. As discussed above, we may assume that $N_1$, $N_2$, and $N_1\cap N_2$ are simply connected and therefore orientable.

Observe that we have $3k_1+k_2\leq n$ in all three cases. By the corollary to the connectedness theorem, $H^*(N_2;\Z)$ is $k_1$-periodic since the intersection is transverse. The bound on the codimensions implies $3k_1 \leq \dim(N_2)$, hence Theorem \ref{THM4-newer} implies $H^*(N_2;\Q)$ is $g$-periodic, where $g=\gcd(4,k_1)$. Since $N_1\cap N_2\embedded N_2$ is $\dim(N_1\cap N_2)$-connected, $H^*(N_1\cap N_2;\Q)$ is $g$-periodic as well. This concludes the proof of the third statement.

Assume now that $2k_1 + 2k_2 \leq n$. We claim that $H^*(N;\Q)$ is $g$-periodic. Let $j:N_2\embedded N$ be the inclusion map, and let $x'\in H^g(N_2;\Q)$ be an element inducing periodicity in $H^*(N_2;\Q)$. Because $n-2k_2 \geq 2k_1 \geq g$, the connectedness theorem implies $j^*:H^g(N;\Q) \to H^g(N_2;\Q)$ is an isomorphism. Let $x\in H^g(N;\Q)$ satisfy $j^*(x) = x'$. By the first part of the connectedness theorem, the inclusion $N_1\embedded N$ is $n-k_1-(k_1-1)$ connected, so there exists $e_1\in H^{k_1}(N;\Z)$ such that the maps $H^i(N;\Z) \to H^{i+k_1}(N;\Z)$ given by $y\mapsto e_1 y$ are isomorphisms for $k_1\leq i\leq n-2k_1$. Denote by $\overline{e}_1$ the image of $e_1$ under the natural map $H^{k_1}(N;\Z)\to H^{k_1}(N;\Q)$, and note that $\overline{e}_1$ satisfies the corresponding property with rational coefficients.

Note that $\overline{e}_1$ is some nonzero multiple of $x^{k_1/g}$. This follows since $(x')^{k_1/g}$ generates $H^{k_1}(N_2;\Q)$, and since $j$ is $n-2k_2+1 \geq 2k_1+1$ connected. Replacing $\overline{e}_1$ by any nonzero multiple preserves the multiplicative property of $\overline{e}_1$, so suppose without loss of generality that $\overline{e}_1 = x^{k_1/g}$. We prove in three cases the claim that multiplication by $x$ induces isomorphisms $H^i(N;\Q)\to H^{i+g}(N;\Q)$ for all $0<i<n-g$:
	\begin{itemize}
	 \item For $0<i<2k_1-g$, the claim follows since $N_2\embedded N$ is $n-2k_2$ connected and $n-2k_2+1\geq 2k_1+1$. For $n-2k_1<i<n-g$, one uses the cap product isomorphisms given by $x$ together with Poincar\'e duality to conclude the claim.
	 \item For $2k_1 < i < n-2k_1-g$, one chooses $l\geq 1$ such that $k_1 < i-lk_1 \leq 2k_1$ and uses the fact that multiplication by $\overline{e}_1$ induces isomorphisms in some middle degrees and that $\overline{e}_1$ and $x$ commute.
	  \item For $2k_1-g\leq i\leq 2k_1$ or $n-2k_1-g\leq i\leq n-2k_1$, one factors multiplication by $\overline{e}_1$ as multiplication by $x^{k_1/g - 1}$ followed by multiplication by $x$ and uses the previous two cases.
	\end{itemize}
Hence $x$ induces g-periodicity in $N$, as claimed.

Next let $g' = g$ if $k_1 + 3k_2 > n$ and $g' = \gcd(4,k_1,k_2)$ if $k_1 + 3k_2\leq n$. Our proof will be complete once we show that $N$, $N_1$, $N_2$, and $N_1\cap N_2$ are $g'$-periodic. First, we claim that $N$ is $g'$-periodic.

If $k_1 + 3k_2 > n$, then $H^*(N;\Q)$ is already $g'$-periodic. Suppose then that $k_1+3k_2\leq n$. By the corollary to the connectedness theorem, there exists $e_2\in H^{k_2}(N;\Q)$ such that the maps $H^i(N;\Q) \to H^{i+k_2}(N;\Q)$ induced by multiplication by $e_2$ are isomorphisms for $k_2 \leq i\leq n - 2k_2$. Given that $x$ and $e_2$ commute, we conclude that $e_2$ induces periodicity in $H^*(N;\Q)$. Indeed, suppose $0< i<k_2$. Choose $j\geq 0$ with $k_2\leq i + jg\leq n-2k_2$. Observe that $H^i(N;\Q)\to H^{i+k_2}(N;\Q)$, induced by multiplication by $e_2$, composed with the isomorphism $H^{i+k_2}(N;\Q) \to H^{i+k_2+jg}(N;\Q)$, induced by multiplication by $x^j$, is the same as the composition of isomorphisms
	\[H^i(N;\Q) \to H^{i+jg}(N;\Q) \to H^{i+jg+k_2}(N;\Q)\]
induced by multiplying in the other order. It follows that multiplication by $e_2$ induces isomorphisms $H^i(N;\Q)\to H^{i+k_2}(N;\Q)$ for $0<i<k_2$. Checking the other required periodicity conditions requires similar arguments. Hence we have that $H^*(N;\Q)$ is $g'$-periodic.

Using this periodicity, we now conclude that the rational cohomology rings of $N_1$, $N_2$, and $N_1\cap N_2$ are $g'$-periodic.

First, since $4k_1 \leq 2k_1 + 2k_2 \leq n$, $N_1\embedded N$ induces isomorphisms on cohomology up to half of the dimension of $N_2$. Using Poincar\'e duality, it follows from the fact that $N$ is rationally $g'$-periodic that $N_1$ is too.

Second, observe that $N_2\embedded N$ is $n-2k_2+1 \geq 2k_1 + 1$ periodic. Hence $H^*(N_2;\Q)$ is both $g$-periodic and $g'$-periodic up to degree $2k_1$ (which is at least twice $g$). Since $N_2$ is rationally $g$-periodic, it follows that $N_2$ is rationally $g'$-periodic by arguments similar to those above.

Finally, $N_1\cap N_2\embedded N_2$ is $\dim(N_1\cap N_2)$-connected, so $N_1\cap N_2$ is clearly $g'$-periodic as well. This concludes the proof of the \THM{THMstrongperiodicity}.

\bigskip
\section{Proof of \THM{chi>0}}\label{ProofTHMchi>0}
\bigskip

Before we begin, we state three well known theorems for easy reference:
\begin{theorem}[Berger]\label{Berger}
If $T$ is a torus acting by isometries on an compact, even-dimensional, positively curved manifold $M$, then the fixed-point set $M^T$ is nonempty.
\end{theorem}
\begin{theorem}[Lefschetz]\label{Lefschetz}
If $T$ is a torus acting on a manifold $M$, then the Euler characteristic satisfies $\chi(M) = \chi(M^T)$.
\end{theorem}
\begin{theorem}[Conner]\label{Conner} If $T$ is a torus acting on a manifold $P$, then the sum of the odd Betti numbers satisfies $\bodd(P^T) \leq \bodd(P)$.
\end{theorem}

We also pause to make a definition:
\begin{definition}\label{DEFdimker}
Let $T$ be a torus acting effectively on a manifold $M$. If the $T$-action restricts to an action on a submanifold $N\subseteq M$, let $\dim\ker\of{T|_N}$ denote the dimension of the kernel of the induced action on $N$. 
\end{definition}

Observe that if $\dim\ker\of{T|_N} = d$, then one can find a codimension $d$ subtorus $T'\subseteq T$ whose Lie algebra is complementary to the kernel of the induced $T$-action on $N$. It follows that $T'$ acts almost effectively on $N$.

We recall the setup of \THM{chi>0}. We are given a closed, positively curved Riemannian manifold $M$ with $\dim(M)\equiv 0 \bmod{4}$, and we have an effective, isometric action by a torus $T$ with $\dim(T)\geq 2\log_2(\dim M)$. By \THMS{Berger}{Lefschetz}, $M^T$ is nonempty and $\chi(M)=\chi(M^T)$, hence it suffices to show that $\bodd(F)=0$ for all components $F$ of $M^T$.

Fix a component $F$ of $M^T$. Our goal is to find a submanifold $P$ with $F\subseteq P\subseteq M$ and $\bodd(P)=0$ such that $T$ acts on $P$. It would follow that $F$ is a component of $P^T$, and hence \THM{Conner} would imply $\bodd(F) = 0$.

In order to find such a submanifold $P$, the first step is to set up a sort of induction argument. To do this, we look at our situation from the point of view of $F$. We consider all closed, totally geodesic submanifolds $N^n$ with $F\subseteq N^n\subseteq M$ such that

\begin{center}
$(*)~~\left\{\begin{tabular}{l}~\\~\end{tabular}\right.$\hspace{-.2in}
\begin{tabular}{l}
$n\equiv 0\bmod{4}$, and there exists a subtorus $T'\subseteq T$ acting\\ almost effectively on $N$ with $\dim(T') \geq 2\log_2(n)$.
\end{tabular}
\end{center}

Clearly $M$ satisfies property $(*)$ by the assumption in \THM{chi>0}, so the collection of submanifolds $N$ satisfying $(*)$ is nonempty. As it will simplify later arguments, we complete the induction setup by choosing a submanifold $N^n$ with minimal $n$ satisfying $F\subseteq N\subseteq M$ and property $(*)$.

We make a few remarks before continuing with the proof. First, observe that $F$ is a component of the fixed-point set $N^{T'}$. This follows since $F\subseteq N\subseteq M$ and $T'$ acts almost effectively on $N$. Since our goal is to show $\bodd(F) = 0$, and since we will do this by finding a submanifold $F\subseteq P\subseteq N$ on which $T'$ acts with $\bodd(P)=0$, we may forget about $M$ and $T$ and instead focus on $N$ and $T'$.

Second, since we are focused only on the action of $T'$ on $N$, we may divide $T'$ by its discrete ineffective kernel to assume without loss of generality that $T'$ acts effectively on $N$. Third, it will be convenient to adopt the following notation (recall that $F$ is fixed):
\begin{definition}~
\begin{enumerate}
\item For a submanifold $N'\subseteq N$, let $\cod(N')$ denote the codimension of $N'$ in $N$.
\item For a subgroup $H\subseteq T'$, let $F(H)$ denote the component of the fixed-point set $N^H$ of $H$ which contains $F$. If $H$ is generated by $\sigma\in T'$, we will write $F(\sigma)$ for $F(H)$.
\end{enumerate}
\end{definition}

Finally, the following lemma is a consequence of our choice of $N$. It is one of the two places where the logarithmic bound appears. We will refer to it frequently.
\begin{lemma}\label{InductionLemma}
For a non-trivial subgroup $H\subseteq T'$ with $\dim F(H) \equiv 0\bmod{4}$, we have $\dim F(H) > n/2^{d/2}$ where $d = \dk{T'|_{F(H)}}$.
\end{lemma}
\begin{proof}
Indeed, if $\dim F(H) \leq n/2^{d/2}$, then the remark following \DEF{DEFdimker} implies the existence of a codimension $d$ subtorus $T''\subseteq T'$ acting almost effectively on $F(H)$ with
	\[\dim(T'') = \dim(T') - d \geq 2\log_2(n) -  d \geq 2\log_2(\dim F(H)).\]
Hence $F(H)$ satisfies property $(*)$. But since the action of $T'$ on $N$ is effective, $\dim F(H) < n$, a contradiction to our choice of $N$.
\end{proof}

We now proceed with the second part of the proof, in which we study the array of intersections of fixed-point sets of involutions in $T'$. The strategy is to find $F\subseteq P\subseteq N$ such that $\dim(P) \equiv 0\bmod{4}$ and such that $P$ contains a pair of transversely intersecting submanifolds. This takes work and is the heart of the proof. Once we find this transverse intersection, we will apply \LEM{InductionLemma} to show that the two codimensions are small enough so that the periodicity theorem applies. \COR{CORtoPerThm} will then imply $\bodd(P) = 0$, as required.

To organize the required intersection, codimension, and symmetry data, we define an abstract graph which simplifies the picture while retaining this information:

\begin{definition}
Define a graph $\Gamma$ by declaring the following
\begin{itemize}
\item An involution $\sigma\in T'$ is in $\Gamma$ if $\cod F(\sigma) \equiv_4 0$ and $\dim\ker\of{T'|_{F(\sigma)}}\leq 1$, and
\item An edge exists between distinct $\sigma,\tau\in\Gamma$ if $F(\sigma)\cap F(\tau)$ is not transverse.
\end{itemize}
\end{definition}

We are ready to prove the existence of a submanifold $F\subseteq P\subseteq N$ on which $T'$ acts with $\bodd(P)=0$. As we will see, the $P$ we choose will be $N$ itself or $F(H)$ for some $H\subseteq T'$, so we will only need to show that $\bodd(P) = 0$. We separate the proof into five cases, according to the structure of $\Gamma$. 

\begin{lemma}[Case 1]
Let $r = \dim(T')$. If $\Gamma$ does not contain $r-1$ algebraically independent involutions, then $\bodd(N) = 0$.
\end{lemma}

\begin{proof}
Let $0\leq j\leq r-2$ be maximal such that there exist $\iota_1,\ldots,\iota_j\in\Gamma$ generating a $\Z_2^j$. We wish to show that $\bodd(N)=0$.

Consider the isotropy $T'\embedded SO(T_x N)$ for some $x\in F$. Choose a basis of the tangent space so that the image of the $\Z_2^r \subseteq T'$ lies in a copy of $\Z_2^{n/2}\subseteq T^{n/2}\subseteq SO(T_x N)$. Let $\Z_2^{r-1}$ denote a subspace of the kernel of the composition $\Z_2^r \to \Z_2^{n/2} \to \Z_2$, where the last map is given by $(\tau_1,\ldots,\tau_{n/2}) \mapsto \sum \tau_i$. It follows that every $\sigma\in \Z_2^{r-1}$ has $\cod F(\sigma) \equiv 0\bmod{4}$.

Since $j\leq r-2$, there exists $\iota_{j+1}\in\Z_2^{r-1}\setminus\langle\iota_1,\ldots,\iota_j\rangle$. Choosing $\iota_{j+1}$ to have minimial $\cod F(\iota_{j+1})$ ensures that $\dim\ker\of{T'|_{F(\iota_{j+1})}}\leq 2$. Moreover, because $j$ is maximal, we cannot have $\iota_{j+1}\in\Gamma$. Hence $\dim\ker\of{T'|_{F(\iota_{j+1})}} = 2$. This implies the existence of an involution $\iota\in T'$ such that $F(\iota_{j+1})\subseteq F(\iota) \subseteq M$ with all inclusion strict. It follows that $F(\iota_{j+1})$ is the transverse intersection of $F(\iota)$ and $F(\iota\iota_{j+1})$. Moreover, \LEM{InductionLemma} implies $\dim F(\iota_{j+1}) > \frac{n}{2}$, which implies
	\[2\cod F(\iota) + 2\cod F(\iota\iota_1) = 2\cod F(\iota_{j+1}) < n.\]
By \COR{CORtoPerThm}, $\bodd(N) = 0$.
\end{proof}

\begin{lemma}[Case 2]
If there exist distinct $\sigma,\tau\in\Gamma$ such that $\dim\ker\of{T'|_{F(\langle\sigma,\tau\rangle)}} \geq 3$, then $\bodd(F(\tau)) = 0$.
\end{lemma}

\begin{proof}
Let $H=\langle\sigma,\tau\rangle$. Since $\dim\ker\of{T'|_{F(H)}} \geq 3$ and $\dim\ker\of{T'|_{F(\tau)}}\leq 1$, there exists a 2-torus that acts almost effectively on $F(\tau)$ and fixes $F(H)$. Restricting our attention to the action on $F(\tau)$, we may divide by the kernel of this action to conclude that a 2-torus acts effectively on $F(\tau)$ and fixes $F(H)$. This implies the existence of an involution $\iota$ such that $F(H) \subseteq F(\iota) \subseteq F(\tau)$ with all inclusions strict. Since $F(H)$ is the $F$-component of the fixed-point set of the $\sigma$-action on $F(\tau)$, it follows that $F(H)$ is the transverse intersection inside $F(\tau)$ of $F\of{\iota|_{F(\tau)}}$ and $F\of{\iota\sigma|_{F(\tau)}}$. 

\LEM{InductionLemma} implies $\dim F(\tau) \geq n/\sqrt{2} > 2n/3$ and similarly for $\dim F(\sigma)$. Hence
	\[\cod_{F(\tau)} F(\iota|_{F(\tau)}) + \cod_{F(\tau)} F(\iota\sigma|_{F(\tau)})
		= \cod_{F(\tau)} F(\sigma|_{F(\tau)}) \leq \cod F(\sigma) < \frac{1}{2}\dim F(\tau).\]
\COR{CORtoPerThm}, together with the observation
	\[\dim F(\tau) = n - \cod F(\tau) \equiv 0\bmod{4},\]
therefore implies $\bodd(F(\tau))=0$.
\end{proof}

\begin{lemma}[Case 3]
If there exist distinct $\sigma,\tau\in\Gamma$ with no edge connecting them, then $\bodd(F(\tau))=0$.
\end{lemma}

\begin{proof}
Let $H=\langle\sigma,\tau\rangle$. By the proof in Case 2, we may assume that $\dim\ker\of{T'|_{F(H)}}\leq 2$. By \LEM{InductionLemma}, therefore, $\dim F(H) > n/2$. The assumption that no edge exists between $\sigma$ and $\tau$ means that $F(\sigma)\cap F(\tau)$ is transverse. Since
	\[2\cod F(\sigma) + 2\cod F(\tau) = 2\cod F(H) < n,\]
\COR{CORtoPerThm} implies $\bodd(F(\tau))=0$.
\end{proof}

\begin{lemma}[Case 4]
If there exist distinct $\sigma,\tau\in\Gamma$ such that $\sigma\tau\not\in\Gamma$, then $\bodd(F(\tau))=0$ or $\bodd(N)=0$.
\end{lemma}

\begin{proof}
It follows from the isotropy representation that
	\[\cod F(\sigma\tau) \equiv \cod F(\sigma) + \cod F(\tau)\]
modulo 4, so $\sigma\tau\not\in\Gamma$ implies $\dim\ker\of{T'|_{F(\sigma\tau)}}\geq 2$. On the other hand, the fact that $F(H) \subseteq F(\sigma\tau)$ implies
	\[\dim\ker\of{T'|_{F(\sigma\tau)}} \leq \dim\ker\of{T'|_{F(H)}},\]
and the proof in Case 2 implies that we may assume
	$\dk{T'|_{F(H)}} \leq 2,$
hence we have $\dim\ker\of{T'|_{F(\sigma\tau)}}=2$.

This implies the existence of an involution $\rho\in T'$ satsifying $F(\sigma\tau)\subseteq F(\rho)\subseteq M$ with all inclusions strict, which in turn implies $F(\sigma\tau)$ is the transverse intersection in $M$ of $F(\rho)$ and $F(\rho\sigma\tau)$. Additionally $\dim\ker\of{T'|_{F(\sigma\tau)}}=2$ implies $\dim F(\sigma\tau) > n/2$ by \LEM{InductionLemma}. Hence
	\[2\cod F(\rho) + 2\cod F(\rho\sigma\tau) = 2\cod F(\sigma\tau) < n,\]
so the periodicity theorem implies $\bodd(N) = 0$.
\end{proof}

We pause before considering the last case. By the proof of Cases 3 and 4, we may assume that $\Gamma$ is a complete graph and that the set of vertices in $\Gamma$ is closed under multiplication. Adding the proof of Case 1, we may assume, in fact, that $\Gamma$ is a complete graph on $\Z_2^m$ for some $m\geq \dim(T')-1$. The last case considers this possibility.

We make a minor modification to our definition of $F(\rho)$. If $\rho\in T'$ and $H\subseteq T'$, then $\rho$ acts on $F(H)$. Let $F(\rho|_{F(H)})$ denote the $F$-component of the fixed-point set of the $\rho$-action on $F(H)$. Observe that $F(\rho|_{F(H)}) = F(H')$ where $H'$ is the subgroup generated by $\rho$ and $H$.

\begin{lemma}[Case 5]
Suppose $\Gamma$ is a complete graph on $\Z_2^m$ with $m\geq \dim(T')-1$. There exists $H\subseteq T'$ such that $\bodd(F(H)) = 0$.
\end{lemma}

\begin{proof}
Set $l = \floor{\frac{m+1}{2}}$. Choose subgroups
	\[\Z_2^m\supseteq\Z_2^{m-1}\supseteq\cdots\supseteq \Z_2^{m-(l-1)}\]
and 
	\[\rho_i \in \Z_2^{m-(i-1)} \setminus\langle\rho_1,\ldots,\rho_{i-1}\rangle\]
for $1\leq i\leq l$ according to the following procedure:
	\begin{itemize}
	\item Choose $\rho_1\in\Z_2^m$ such that $k_1 = \cod F(\rho_1)$ is maximal.
	\item Given $\Z_2^m\supseteq\cdots\supseteq\Z_2^{m-(i-1)}$ and $\rho_j\in\Z_2^{m-(j-1)}$ for $1\leq j\leq i$, choose $\Z_2^{m-i}\subseteq \Z_2^{m-(i-1)}$ such that every $\rho\in\Z_2^{m-i}$ satisfies
		\[\cod\of{F(\rho|_{R_i}) \embedded R_i} \equiv 0\bmod{4}\]
	where $R_i = F\of{\langle\rho_1,\ldots, F(\rho_i)\rangle}$, and then choose $\rho_{i+1}\in\Z_2^{m-i}$ such that
		\[\cod\of{F\of{\rho_{i+1}|_{R_i}} \embedded R_i} = k_{i+1}\]
	is maximal.
	\end{itemize}
We claim that our choices imply
	\begin{enumerate}
	\item $\dim(R_h) \equiv 0 \bmod{4}$ for all $h$,
	\item $k_h \geq 2k_{h+1}$ for all $h$, and
	\item $k_l = 0$.
	\end{enumerate}
The first point follows by observing that $\dim(R_h) = n - (k_1 + \ldots + k_h)$ by definition and that $k_i \equiv 0\bmod{4}$ for all $i$ by our choices.

To prove the second claim, fix $h\geq 1$. Observe that $\rho_h\in \Z_2^{m-(h-1)}$ and $\rho_{h+1}\in \Z_2^{m-h} \subseteq \Z_2^{m - (h-1)}$, so $\rho_h\rho_{h+1}\in\Z_2^{m-(h-1)}$ as well. By maximality then, we have
	\begin{eqnarray*}
	k_h &\geq& \cod      \of{F\of{\rho_{h+1}|{R_{h-1}}}\embedded R_{h-1}} = k_{h+1} + a, \mathrm{~and}\\
	k_h &\geq& \cod\of{F\of{\rho_h\rho_{h+1}|R_{h-1}}\embedded R_{h-1}} = k_{h+1} + (k_h - a)
	\end{eqnarray*}
where $a = \cod\of{R_{h+1} \embedded F\of{\rho_h\rho_{h+1}}|_{R_{h-1}}}$. Adding these inequalities shows that $k_h \geq 2k_{h+1}$.

Finally, the third claim follows from the second claim together with the estimate
	\[l = \floor{\frac{m+1}{2}} \geq \floor{\frac{\dim(T')}{2}} > \log_2(n) - 1\]
and the fact that $k_1 < n/\sqrt{2}$ by \LEM{InductionLemma}.

We now use these facts to find a transverse intersection. Let $0<j\leq l$ be the smallest index such that $k_j = 0$. For $1\leq i \leq j-1$, let $l_i$ be the number of $(-1)$s in the image of $\rho_j$ in $SO(T_x R_{i-1}\cap \nu_x R_i)$. Geometrically, $l_i$ is the codimension of 
	\[F(\rho_i|_{R_{i-1}})\cap F(\rho_j|_{R_{i-1}})\embedded F(\rho_i\rho_j|_{R_{i-1}}).\]
By replacing $\rho_j$ by $\rho_{j-1}\rho_j$ if necessary, we can ensure that $l_{j-1} \leq \frac{k_{j-1}}{2}$. Observe that this may change $l_i$ for $i<j-1$. Next, replace $\rho_j$ by $\rho_{j-2}\rho_j$ if necessary to ensure that $l_{j-2} \leq \frac{k_{j-2}}{2}$. Observe again that the $l_i$ may have changed for $i<j-2$, but that $l_{j-1}$ does not. Continuing in this way, we may replace $\rho_j$ by $\rho\rho_j$ for some $\rho\in\langle\rho_1,\ldots,\rho_{j-1}\rangle$ to ensure that $l_i \leq \frac{k_i}{2}$ for all $i<j$. 

Now some of the $l_{j-1},l_{j-2},\ldots$ may be zero, but they cannot all be zero because the action of $T'$ is effective and $\rho_j\not\in\langle\rho_1,\ldots,\rho_{j-1}\rangle$. Let $1\leq i\leq j-1$ denote the largest index where $l_i>0$. Observe that $l_i\leq \frac{k_i}{2}$ implies $l_i>0$ and $k_i-l_i>0$.

Consider now the transverse intersection of $F(\rho_j|_{R_{i-1}})$ and $F(\rho_j\rho_i|_{R_{i-1}})$ inside $R_{i-1}$. The intersection is $R_i$ (by choice of $i$), and the codimensions are $l_i$ and $k_i-l_i$. We wish to apply \COR{CORtoPerThm} to this intersection to conclude $\bodd(R_{i-1})=0$. 

First, observe that $k_1 \leq n-\frac{n}{\sqrt{2}} < n/2$ by \LEM{InductionLemma}. Also recall that $k_h\geq 2k_{h+1}$ for all $h$. Hence
	\[2l_i + 2(k_i - l_i) = 2k_i \leq \frac{n}{2^{i-1}} = n - \sum_{h=1}^{i-1} \frac{n}{2^h} \leq n-\sum_{h=1}^{i-1} k_h = \dim(R_{i-1}).\]
Moreover $\dim R_{i-1} \equiv 0\bmod{4}$, so \COR{CORtoPerThm} implies $\bodd(R_{i-1})=0$. Since $R_{i-1} = F(H)$ where $H = \langle\rho_1,\ldots,\rho_{i-1}\rangle$, this concludes the proof in this case.
\end{proof}

We have shown in all five cases the existence of a submanifold $F\subseteq P\subseteq N$ on which $T'$ acts such that $\bodd(P)=0$. As explained at the beginning of the proof, Conner's theorem then implies $\bodd(F) = 0$, as required.

\bigskip
\section{A Corollary and a Conjecture}\label{Observations}
\bigskip

Our first point of discussion regards general Lie group actions. By examining the list of simple Lie groups, one easily shows that $\of{2\rank(G)}^2 \geq \dim(G)$ for all compact, 1-connected, simple Lie groups. The inequality persists for all compact Lie groups. In addition, $\dim(M^n/G)\leq n-d$ clearly implies $\dim(G)\geq d$. Hence, letting $I(M)$ denote the isometry group of $M$, we have the following corollary:

\begin{corollary}\label{CORchi>0}
Let $M^n$ be a closed Riemannian manifold with positive sectional curvature and $n\equiv 0\bmod{4}$. If $\dim I(M) \geq \of{4\log_2 n}^2$ or $\dim M/I(M) \leq n - \of{4\log_2 n}^2$, then $\chi(M)>0$.
\end{corollary}

We remark that in \cite{Puettmann-Searle02} it was shown that $\chi(M^{2n}) > 0$ if $\dim M/I(M) < 6$, and in \cite{Wilking06} it was shown that $\chi(M^{2n})>0$ if $\dim M/I(M) \leq \sqrt{n}/3 - 1$ or $\dim I(M) \geq 4n-6$.

To conclude, we state a conjecture which would improve the conclusion of the periodicity theorem. Recall that the periodicity theorem rested on \PROPS{mod2lemma}{modplemma}, which we referred to as generalizations of Adem's theorem on singly generated cohomology rings. The conclusion of Adem's theorem was improved by Adams after he developed the theory of secondary cohomology operations (see \cite{Adams60}). The result is the following:
\begin{theorem}[Adams]
Let $p$ be a prime, and let $M$ be a topological space. Assume $H^*(M;\Z_p)$ is isomorphic to $\Z_p[x]$ or $\Z_p[x]/x^{q+1}$ with $p\leq q$.
\begin{enumerate}
 \item If $p=2$, then $k\in\{1,2,4,8\}$. Moreover, $k = 8$ only occurs when $q=2$.
 \item If $p>2$, then $k = 2\lambda$ for some $\lambda|p-1$.
\end{enumerate}

\end{theorem}

Observe that singly generated cohomology rings are periodic in the sense of this paper. The corresponding strengthening in our case would be the following:

\begin{conjecture}
Let $p$ be a prime, and let $M$ be a topological space. Assume $x\in H^k(M;\Z_p)$ is nonzero and induces periodicity up to degree 
$pk$, and suppose $x$ has minimal degree among all such elements.
\begin{enumerate}
 \item If $p=2$, then $k\in\{1,2,4,8\}$. Moreover, if $x$ induces periodicity up to degree $3k$, then $k\neq 8$.
 \item If $p>2$, then $k = 2\lambda$ for some $\lambda |p-1$.
\end{enumerate}
\end{conjecture}

We first note that, regarding the first statement, 
$S^{k-1}\times S^k$ is $k$-periodic but not $k'$-periodic for any $k'<k$, and $S^7\times \mathrm{Ca}P^2$ is $8$-periodic but not 4-periodic. Hence one must assume periodicity up to degree $2k$, respectively $3k$.

Second, we wish to outline how a proof of this conjecture would imply that \THM{chi>0} holds in all even dimensions. First, one would use the conjecture to improve \PROP{mod2lemma} to prove the following: If $M$ is a simply connected, closed manifold such that $H^*(M^n;\Z_2)$ is $k$-periodic with $3k\leq n$, then $M$ has the $\Z_2$-cohomology ring of $S^n$, $\C P^{n/2}$, $\HH P^{n/4}$, $\HH P^{(n-3)/4}\times S^3$, or $\HH P^{(n-2)/4}\times S^2$. Indeed, a proof of the $\Z_2$-periodicity conjecture combined with Poincar\'e duality implies this when $n\not\equiv 2\bmod{4}$.

Suppose then that $n\equiv 2\bmod{4}$. We may assume without loss of generality that $H^4(M;\Z_2)\cong\Z_2$ and that the generator $x$ has minimal degree among all elements inducing periodicity. It follows that $Sq^1(H^3(M;\Z_2)) = 0$, $Sq^1(H^7(M;\Z_2)) = 0$, and $Sq^2(H^2(M;\Z_2)) = 0$. 

By periodicity and Poincar\'e duality, $H^2(M;\Z_2)\cong \Z_2$. Let $z_2\in H^2(M;\Z_2)$ be a generator. If $H^3(M;\Z_2)=0$, it follows that $H^*(M;\Z_2)\cong H^*(S^2\times \HH P^{(n-2)/4};\Z_2)$. To see that this is the case, suppose there exists a nonzero $u \in H^3(M;\Z_2)$. Using Poincar\'e duality and periodicity again, we conclude the existence of a relation $uv = xz$ for some $v\in H^3(M;\Z_2)$. One can now use the Cartan formula to prove that $Sq^4(uv) = 0$ and $Sq^4(xz) = x^2z \neq 0$, which is a contradiction.

Given this, the basic outline of our proof of \THM{chi>0} implies the result without the assumption that the dimension is divisible by four. In fact, the proof simplifies since one does not have to keep track of the divisibility of the codimensions. The optimal bound, as far as the proof is concerned, would be $r \geq \log_2(n) - 2$.

\bibliographystyle{amsplain}
\bibliography{myrefs}
\end{document}